\documentclass{amsart}

\usepackage{amssymb,bm,amsthm,color,enumerate}
\usepackage[dvipdfm]{hyperref}
 \newtheorem{theorem}{Theorem}[section]
 
 \newtheorem*{thghjkl: eorem*}{Theorem}
 \newtheorem{proposition}[theorem]{Proposition}
 \newtheorem{lemma}[theorem]{Lemma}

\theoremstyle{definition}
 \newtheorem{definition}[theorem]{Definition}
\theoremstyle{remark}
 \newtheorem{remark}[theorem]{Remark}
 \newtheorem{example}[theorem]{Example}
\numberwithin{equation}{section}

\newcommand{\qtbinom}[2]{\genfrac{[}{]}{0pt}{}{#1}{#2}}

\title[Erd\H{o}s--Ko--Rado theorem to $t$-designs]{A generalization of Erd\H{o}s--Ko--Rado theorem to $t$-designs in certain semilattices}
\author{Sho Suda}
\address{Graduate School of Information Sciences, Tohoku University, 6-3-09 Aramaki-Aza-Aoba, Aoba-ku, Sendai 980-8579, Japan}
\email{suda@ims.is.tohoku.ac.jp}
\date{\today}
\begin{document}

\begin{abstract}
The Erd\H{o}s--Ko--Rado theorem is extended to designs in semilattices with certain conditions.
As an application, we show the intersection theorems for the Hamming schemes, the Johnson schemes, bilinear forms schemes, Grassmann schemes, signed sets, partial permutations and restricted signed sets.
\end{abstract}
\maketitle

\section{Introduction}
The intersection theorem was shown in 1961 by Erd\H{o}s, Ko and Rado \cite{EKR1961} for a subset in the Johnson scheme.
A set of $m$-subsets in a set of $v$ elements is said to be an $s$-intersection family if the size of intersection is at least $s$ for any two $m$-subsets in the set of $m$-subsets.
The Erd\H{o}s-Ko-Rado theorem states that if $v$ is large enough compared to $m$ and $s$, then the size of an $s$-intersection family is at most $\tbinom{v-s}{m-s}$ and equality holds if and only if the $s$-intersection family consists of all $m$-subsets containing an $s$-subset.   

There are many generalisations of the intersection theorem for $P$- and $Q$-polynomial association schemes \cite{FW1986,H1975,H1987,M1982,T2011} and some combinatorial objects \cite{BL1997,KL2006,Y2010}.
For $P$ and $Q$-polynomial association schemes including Hamming, Johnson, bilinear,  Grassmann and twisted Grassmann schemes, an $s$-intersection family can be regarded as a subset with the width at most $m-s$, where $m$ is the number of classes of the association scheme.
Applying the linear programing, the size of an $s$-intersection family is bounded above, with equality if and only if the subset has the dual width $s$.
Here, width and dual width are parameters defined for a subset of the vertex set in a $P$ and $Q$-polynomial association scheme \cite{BGKM2003}.
See \cite{T2011} for more details on the LP method to the Erd\H{o}s--Ko--Rado theorem.
Furthermore, for the last decade, several Erd\H{o}s--Ko--Rado type theorems have been shown for maps on finite sets with certain conditions.

In 1982 Rands extended the intersection theorem for designs in the Johnson schemes.
If we regard the whole vertex set of the Johnson scheme as a design, then the intersection theorem is obtained from Rands' result as a corollary.
Later in 1999, Fu \cite{F1999} extended Rands' method to regular quantum matroids, defined by Terwilliger \cite{T1996}.
By this extension, Fu obtained the results for intersection theorems in $t$-designs in bilinear forms schemes and Grassmann schemes. 

It was shown by Delsarte \cite{D1976} that the concept of a design can be defined for a subset in the top fiber of a regular semilattice.
In fact, for block designs or orthogonal arrays, there are some regular semilattices in which these can be defined and coincide with the designs defined by Delsarte.

In this paper we extend Rands method to designs in a class of semilattices with certain regularity conditions.
We set the intersection property for a subset in the top fiber of a semilattice and   
the above objects which are considered for the intersection theorem appear in the top fiber of semilattices.

Our main theorem is a two-step generalization of the Erd\H{o}s--Ko--Rado result.
The first is to consider intersection families in a design in a semilattice.
The second is to unify the proof the Erd\H{o}s--Ko--Rado theorem for known examples.
To characterize an $s$-intersection family attaining the upper bound, we will give sufficient conditions on the parameters of the semilattice.

\section{Preliminaries}
Let $X$ be a finite set and $\preceq$ a partial order on $X$.
A partially ordered set $(X,\preceq)$ is said to be a semilattice if for any two points $x,y\in X$ there exists an unique point $z\in X$ such that the greatest lower bound of $x$ and $y$.
Such a $z$ is denoted by $x\wedge y$.
Denote the least element of $X$ by $0$.

We assume that a semilattice $(X,\preceq)$ has a rank function $|\cdot|:X\rightarrow \mathbb{N}\cup\{0\}$ such that $|x|+1$ is the number of terms in a maximal chain from the least element $0$ to the element $x$ including the end elements in the count.
Let $m$ be the maximum value of the rank function.
For any $0\leq i\leq m$, define the fiber $X_i:=\{x\in X: |x|=i\}$.

Throughout this paper, we also assume that a semilattice $(X,\preceq)$ satisfies the following conditions:
\begin{enumerate}[(I)]
\item For $y\in X_m$ and $z\in X_r$, the number of points $u\in X_s$ such that $z\preceq u\preceq y$ is a constant $\mu(r,s)$.
\item For $u\in X_s$, the number of points $z\in X_r$ such that $z\preceq u$ is a constant $\nu(r,s)$.
\item For any $0\leq r \leq m$ and $a\in X_r$, the number of $z\in X_m$ such that $a\preceq z$ is a constant $\theta(r)$.
\item \label{property} If there exists an upper bound of $x$ and $y$, 
then the unique least upper bound is in $X_{i+j-k}$ for elements $x\in X_i$, $y\in X_j$ such that $x\wedge y \in X_k$ and $i+j-k\leq m$, i
\end{enumerate} 
For a subset $Z\subset X$, we denote $Z\wedge Z=\{x\wedge y : x,y\in Z \}$.

We have the following lemma on a regularity condition.
\begin{lemma}\label{lem:a}
For $0\leq r<s \leq m-1$ and $u\in X_r$, the number $z\in X_s$ such that $u\preceq z$ is a constant $\alpha(r,s)$.
\end{lemma}
\begin{proof}
For any element $u\in X_r$, denote by $\alpha_u(r,s)$ the size of $\{z\in X_s : u\preceq z \}$.
Counting pairs $(z,w)\in X_s\times X_m$ such that $u\preceq z \preceq w$ in two ways yields $\alpha_u(r,s)\theta(s)=\theta(r)\mu(r,s)$, which implies that $\alpha_u(r,s)$ is a constant.
\end{proof}
  
We consider the following conditions on subsets in top fibers of semilattices.
\begin{definition}
Let $s$ be a positive integer with $s < m$, $Y$ a non-empty subset in $X_m$.
The set $Y$ is said to be an $s$-intersection family, if the rank of $x \wedge y$ is at least $s$ for any points $x,y\in Y$.
\end{definition}
\begin{definition}
Let $t$ be a positive integer with $t \leq m$, $Y$ a non-empty subset in $X_m$.
The set $Y$ is said to be a $t$-design with index $\lambda_t$ if for any point $z\in X_t$, the size of 
\[
\{x\in Y :  z\preceq x \}
\]   
is a constant $\lambda_t$.
\end{definition}
For an element $z\in X_k$ and a subset $Y\subseteq X_m$ with $k\leq m$, we denote a set $\{x\in Y :  z\preceq x \}$ by $Y_z$.

For any $0\leq t\leq m$, the assumption (III) implies that the set $X_m$ itself is a $t$-design with index $\theta(t)$.
\begin{proposition}\label{prop:design}
A $t$-design $Y$ in $X_m$ with index $\lambda_t$ is also a $t'$-design with index $\lambda_{t'}=\lambda_t\theta(t')/\theta(t)$ where $t' \leq t$.
\end{proposition}
\begin{proof}
For a point $x\in X_{t'}$, define $\lambda_{t'}(x)$ to be the number of $z\in Y$ such that $x\preceq z$.
Counting the pairs $(y,z)\in X_t\times Y$ with $x\preceq y\preceq z$, we have $\lambda_{t'}(x)\mu(t',t)=\alpha(t',t)\lambda_t$ by Lemma~\ref{lem:a}.
Thus $\lambda_{t'}(x)=\lambda_t\theta(t')/\theta(t)$ is a constant. 
\end{proof}

The following are examples of semilattices satisfying the conditions (I)--(IV).
\begin{example}[Johnson scheme]\label{ex:johnson}
Let $v,m$ be integers such that $v\geq 2m\geq 2$.
Let $V=\{1,2,\ldots,v\}$ and $X$ the set of subsets of $V$.
Take $\preceq$ as the usual inclusion.
Then $(X,\preceq)$ forms a semilattice with rank function that counts the number of elements of a finite set.
Parameters are shown to be (see \cite[Theorem~5]{D1976}):
\begin{align*}
\mu(r,s)=\tbinom{m-r}{m-s},\quad \nu(r,s)&=\tbinom{s}{r},\quad \theta(r)=\tbinom{v-r}{m-r}. 
\end{align*} 
A $t$-design in the top fiber clearly coincides with a block design with strength $t$.
See \cite{BJL1999-1}, \cite{BJL1999-2} for more information. 
\end{example}

\begin{example}[Grassmann scheme]\label{ex:grassmann}
Let $v,m$ be integers such that $v\geq 2m\geq 2$ and $q$ a prime power.
Let $V$ be a $v$-dimensional vector space over a finite field $GF(q)$ with order $q$ and $X$ the set of subspaces of $V$.
Take $\preceq$ as the usual inclusion.
Then $(X,\preceq)$ forms a semilattice with rank function  taking the dimension of a vector space.
Parameters are shown to be  (see \cite[Theorem~5]{D1976}):
\begin{align*}
{\textstyle \mu(r,s)= \qtbinom{m-r}{m-s}_q},\quad {\textstyle \nu(r,s)=\qtbinom{s}{r}_q},\quad  {\textstyle \theta(r)}&{\textstyle =\qtbinom{v-r}{m-r}_q}. 
\end{align*} 
For $t$-designs see \cite{BKL2005,S1990,S1992,T1987}.
\end{example}

For two sets $A,B$, we denote the set of maps from $A$ to $B$ by $\text{Map}(A,B)$.
\begin{example}[Hamming scheme]\label{ex:hamming}
Let $m,n$ be integers such that $m\geq 1, n\geq 2$.
Let $V$ and $F$ be finite sets with sizes $m$ and $n$ respectively.
Define 
$X=\{(E,f): E\subseteq V, f\in\text{Map}(E,F)\}$.
For $(E,f),(E',f')\in X$, set $(E,f)\preceq (E',f')$ if and only if $E\subseteq E'$ and $f'|_E=f$.
Then $(X,\preceq)$ forms a semilattice with rank function taking the size of $E$.
Parameters are shown to be (see \cite[Theorem~5]{D1976}):
\begin{align*}
\mu(r,s)=\tbinom{m-r}{m-s},\quad \nu(r,s)=\tbinom{s}{r},\quad \theta(r)=n^{m-r}. 
\end{align*} 
A $t$-design in the top fiber clearly coincides with an orthogonal array with strength $t$.
See \cite{HSS1999} for more information on orthogonal arrays.
\end{example}

For two vector spaces $A,B$ over a same field, we denote the set of linear maps from $A$ to $B$ by $\text{Hom}(A,B)$.

\begin{example}[Bilinear forms scheme]\label{ex:bilinear}
Let $m,n$ be integers such that $m,n\geq 1$ and $q$ a prime power.
Let $V$, $F$ be finite dimensional vector spaces over a finite field $GF(q)$ with dimensions $m$ and $n$ respectively.
Define 
$X=\{(E,f): E \text{ is a subspace of } V, f\in\text{Hom}(E,F)\}$.
For $(E,f),(E',f')\in X$, set $(E,f)\preceq (E',f')$ if and only if $E\subseteq E'$ and $f'|_E=f$.
Then $(X,\preceq)$ forms a semilattice with rank function taking the dimension of $E$.
Parameters are shown to be (see \cite[Theorem~5]{D1976}):
\begin{align*}
{\textstyle \mu(r,s)=\qtbinom{m-r}{m-s}}_q,\quad {\textstyle \nu(r,s)=\qtbinom{s}{r}_q},\quad \theta(r)=n^{m-r}. 
\end{align*} 
For $t$-designs, which are known as Singleton systems, see \cite{D1978}.
\end{example}

For two sets $A,B$, we denote the set of injective maps from $A$ to $B$ by $\text{Inj}(A,B)$.
The following is considered in \cite{KL2006}.
\begin{example}[Injection scheme]\label{ex:injection}
Let $m,n$ be integers such that $n\geq m\geq 1$.
Let $V$ and $W$ be finite sets with sizes $m$ and $n$ respectively.
Define 
$X=\{(E,f): E\subset V, F\subset W, f\in\text{Inj}(E,F)\}$.
For $(E,f),(E',f')\in X$, set $(E,f)\preceq (E',f')$ if and only if $E\subset E',f'|_E=f$.
Then $(X,\preceq)$ forms a semilattice with rank function taking the size of $E$.
Parameters are as follows:
\begin{align*}
\mu(r,s)=\tbinom{m-r}{m-s},\quad \nu(r,s)&=\tbinom{s}{r},\quad \theta(r)=\tfrac{(n-r)!}{(n-m)!}. 
\end{align*} 
A $t$-design in the top fiber coincides with an orthogonal array of Type $I$ with strength $t$.
See \cite[p.~132]{HSS1999} for more information on orthogonal arrays of Type $I$.
\end{example}

The following example is considered in \cite{TAG1985}, \cite{BL1997}. 
\begin{example}[Non-binary Johnson, Signed sets]\label{ex:nonbinary}
Let $(X,\preceq)$ be the same as in Example~\ref{ex:hamming}, $k$ be a positive integer with $k<m$.
Recall that $X_i$ is the set of rank $i$ objects in the semilattice $(X,\preceq)$ for $0\leq i\leq m$.
We consider the semilattice $(\cup_{i=0}^{k}X_i,\preceq)$.
Parameters are as follows:
\begin{align*}
\mu(r,s)=\tbinom{k-r}{s-r}, \quad \nu(r,s)=\tbinom{s}{r},\quad \theta(r)=n^{k-r}\tbinom{m-r}{k-r}. 
\end{align*}
\end{example}

The next example is shown in \cite{Y2010}.
\begin{example}[Restricted signed sets]\label{ex:restksign}
Let $m$ be a positive integer.
Let $V$ be a finite set with size $m$.
Define 
$X=\{(E,f): E\subseteq V, |E|\leq k, f\in\text{Map}(E,V),f(i)\neq i\text{ for all }i\in E\}$.
For $(E,f),(E',f')\in X$, set $(E,f)\preceq (E',f')$ if and only if $E\subseteq E'$ and $f'|_E=f$.
Then $(X,\preceq)$ forms a semilattice with rank function taking the size of $E$.
Parameters are as follows:
\begin{align*}
\mu(r,s)=\tbinom{k-r}{s-r}, \quad \nu(r,s)=\tbinom{s}{r},\quad \theta(r)=(m-1)^{k-r}\tbinom{m-r}{k-r}. 
\end{align*}  
\end{example}
\section{Main theorem}
The following lemma will be used in the proof of Theorem~\ref{thm:main}.
\begin{lemma}\label{lem:dr}
Let $r,s,t$ be positive integers such that $0\leq r\leq s<t\leq m$, and $Y$ a $t$-design in $X_m$.
Define $d_{r}=\max |\{z\in Y: x\preceq z, |z\wedge y|\geq s \}|$ where $x\in X_s$ and $y\in Y$ satisfy $r=|x\wedge y|$.
Then the following statements are valid:
\begin{enumerate}
\item If $2s-t\leq r \leq s-1$, then $d_r\leq \mu(r,s)\lambda_{2s-r}$.
\item If $r \leq 2s-t$, then $d_r\leq \mu(r,s)\lambda_{t}$.
\end{enumerate}  
\end{lemma}
\begin{proof}
(1): Let $x\in X_s$ and $y\in Y$ such that $|x\wedge y|=r$.
Define $d_r(x,y)=|\{z\in Y: x\preceq z, |z\wedge y|\geq s \}|$.
Setting $\mu_{r,s}(x\wedge y,y):=\{u\in X_s : x\wedge y\preceq u \preceq y\}$,
we have 
\[
\{z\in Y : x\preceq z, |z \wedge y|\geq s\}= \{z\in Y : x\preceq z, u\preceq z \text{ for some } u\in \mu_{r,s}(x,y)\}
\]
and the size of the latter set is at most $\mu(r,s)\lambda_{2s-r}$.
Thus $d_r(x,y)\leq \mu(r,s)\lambda_{2s-r}$ holds.

(2): 
Let $c_j$ be the maximum number of the elements of $Y$ containing an element in $X_j$
Then $c_j$ is at most $\lambda_{j}$ if $j \leq t$ and $\lambda_{t}$ if $j>t$.
Using the same method as in (1), we have the desired result.
\end{proof}

The following theorem shows an intersection theorem for a design in a semilattice.
\begin{theorem}\label{thm:main}
Let $(X,\preceq)$ be a semilattice satisfying the conditions (I)--(IV) and 
$r,s,t$ integers with $0\leq r$ and $0<s<t\leq m$.
Let $Y$ be a $t$-design in the top fiber $X_m$ and 
$Z$ an $s$-intersection family in $Y$.
Assume that
\begin{enumerate}
\item $\mu(r,s)\nu(s,m)\lambda_{t}<\lambda_{s}$ for all $ r \leq 2s-t$, 
\item $\mu(r,s)\nu(s,m)\lambda_{2s-r}<\lambda_{s}$ for all $2s-t \leq r \leq s-1$.
\end{enumerate}
Then $|Z|\leq \lambda_s$ holds. 
Equality holds if and only if $Z$ is equal to a set $Y_z$ for some $z\in X_s$. 
\end{theorem}
\begin{proof}
Set $Z_s=(Z\wedge Z)\cap X_s$.
For an element $x\in Z_s$, denote by $n_x$ the number of elements in $Z$ that contain $x$. 
Counting the set $\{(x,y)\in Z_s\times Z : x\preceq y \}$ in two ways yields 
\[
\sum_{x\in Z_s}n_x\leq |Z|\nu(s,m).
\]
Next, counting the set $\{(x,y_1,y_2)\in Z_s\times Z\times Z: x\preceq y_1,x\preceq y_2,y_1\neq y_2 \}$ in two ways yields 
\[
\sum_{x\in Z_s}n_x(n_x-1)=\sum_{y_1,y_2\in Z,y_1\neq y_2}\nu(s,|y_1 \wedge y_2|)\geq |Z|(|Z|-1).
\]
If $Z$ is the set of all elements containing an element $x\in X_s$, then clearly $|Z|=\lambda_s$.
If not, namely there exists an element $y\in Z$ such that $x\not\preceq y$ for any element $x\in Y_s$.
For $d_{r}$ defined in Lemma~\ref{lem:dr}, set $d=\max_{r}d_r$.
Then  $n_x\leq d$ holds.
Thus $$|Z|(|Z|-1)\leq \sum_{x\in Z_s}n_x(n_x-1)\leq (d-1)|Z|\nu(s,m),$$
equivalently $|Z|\leq (d-1)\nu(s,m)+1$.
In particular, $|Z|<d\nu(s,m)$ holds.
By the assumptions $(1),(2)$ and Lemma~\ref{lem:dr}, $|Z|< \lambda_s$ holds.
\end{proof}
\begin{remark}
By Proposition~\ref{prop:design} the assumption of Theorem~\ref{thm:main} is equivalent to
\begin{enumerate}
\item $\nu(r,s)\mu(s,m)\theta(t)<\theta(s)$ for all $ r \leq 2s-t$, 
\item $\nu(r,s)\mu(s,m)\theta(2s-r)<\theta(s)$ for all $2s-t \leq r \leq s-1$.
\end{enumerate}
This implies that the assumption does not depend on the indices $\lambda_i$ ($0\leq i\leq t$) of the design appearing in the theorem. 
\end{remark}

If we take an $m$-design $Y$ as the whole set of the top fiber $X_m$, then $\lambda_s=\theta(s)$.
As a corollary of Theorem~\ref{thm:main} we obtain the intersection theorem. 
We list the conditions of parameters to satisfy the assumption of Theorem~\ref{thm:main}:

\begin{table}[h]
\begin{center}
\caption{Required conditions on Theorem~\ref{thm:main}}
\begin{tabular}{c|c}\label{tab:para}
semilattice & condition on parameters \\
\hline
\hline
Johnson scheme  & $v>s+\tbinom{m}{s}(m-s+1)(m-s)$ if $s<t-1$, \\
& $v>s+(m-s)\tbinom{m}{s}^2$ if $s=t-1$\\
\hline
Grassmann scheme & $q^{v-s}-1>\tfrac{(q^{m-s}-1)(q^{m-s-1}-1)}{q-1}\qtbinom{m}{s}_q^2$ if $s<t-1$  ,\\
& $\tfrac{q^{v-s}-1}{q^{m-s}-1}>\qtbinom{m}{s}_q^2$ if $s=t-1$ \\
\hline
Hamming scheme & $n>(m-s+1)\tbinom{m}{s}$ if $s<t-1$   \\
& $n>\tbinom{m}{s}^{2}$ if $s=t-1$ \\
\hline
Bilinear forms scheme & $n>\tfrac{q^{m-s+1}-1}{q-1}\qtbinom{m}{s}_q$ if $s<t-1$   \\
& $q>\qtbinom{m}{s}_q^{2}$ if $s=t-1$ \\
\hline
Injections & $n>s+(m-s+1)\tbinom{m}{s}$ if $s<t-1$   \\
& $n>s+\tbinom{m}{s}^{2}$ if $s=t-1$ \\
\hline
Non-binary Johnson & $n>\tfrac{(k-s+1)(k-s)}{m-s}\tbinom{k}{s}$ if $s<t-1$   \\
& $n>\tfrac{k-s}{m-s}\tbinom{k}{s}^2$ if $s=t-1$ \\
\hline
Restricted signed sets & $(m-1)(m-s)>(k-s+1)(k-s)\tbinom{k}{s}$ if $s<t-1$   \\
& $(m-1)(m-s)>(k-s)\tbinom{k}{s}^2$ if $s=t-1$ \\
\end{tabular}
\end{center}
\end{table}

\begin{remark}
(1) In Example~\ref{ex:injection} if we restrict $n=m$, then the maps are permutations.
As described in Table~\ref{tab:para}, our theorem cannot be applied to this case.
See \cite{CK2003,EFP2011,GM2009,KR2008} for the intersection theorem for permutations.  

(2) The required assumption in Theorem~\ref{thm:main} to obtain the intersecting theorem seems far from the sharp bound.
In fact if we take the top fiber as a design in several cases, the sharp bounds are summarized in \cite[P.3]{T2011}.
\end{remark}
\section*{Acknowledgements}
The author would like to thank the anonymous referees for pointing out errors and many useful suggestions.


\begin{thebibliography}{99}

\bibitem{BJL1999-1}
T.~Beth, D.~Jungnickel, H.~Lenz,  
Design theory.\ Vol. I. Second edition. Encyclopedia of Mathematics and its Applications, 69. Cambridge University Press, Cambridge, 1999. xx+1100 pp.

\bibitem{BJL1999-2}
T.~Beth, D.~Jungnickel, H.~Lenz,  
Design theory. Vol. II. Second edition. Encyclopedia of Mathematics and its Applications, 78. Cambridge University Press, Cambridge, 1999. pp. i--xxi and 60--1100.

\bibitem{BL1997}
B.~Bollobas, I.~Leader,  An Erd\H{o}s-Ko-Rado theorem for signed sets. Graph theory in computer science, chemistry, and other fields (Las Cruces, NM, 1991). Comput. Math.\ Appl.\ 34 (1997), no. 11, 9--13.

\bibitem{BKL2005}
M. Braun, A. Kerber, R. Laue,
Systematic construction of q-analogs of t-(v,k,ă)-designs, 
Des.\ Codes Cryptogr. 34 (2005), no. 1, 55--70. 

\bibitem{BGKM2003}
A.~E.~Brouwer, C.~D.~Godsil, J.~H.~Koolen, W.~J.~Martin, Width and dual width of subsets in polynomial association schemes, 
J. Combin.\ Theory Ser.\ A 102 (2003), 255--271.

\bibitem{CK2003}
P. J. Cameron, C. Y. Ku, 
Intersecting families of permutations, 
European J.\ Combin.\ 24 (2003), no. 7, 881--890. 

\bibitem{D1976}
P.~Delsarte, Association schemes and $t$-designs in regular semilattices,
J. Combin.\ Theory Ser.\ A 20 (1976), no. 2, 230--243.

\bibitem{D1978}
P.~Delsarte, Bilinear forms over a finite field, with applications to coding theory. J. Combin.\ Theory Ser.\ A 25 (1978), no. 3, 226--241.

\bibitem{EFP2011}
D. Ellis, E. Friedgut, H. Pilpel,
Intersecting families of permutations, 
J. Amer.\ Math.\ Soc.\ 24 (2011), no. 3, 649--682. 

\bibitem{EKR1961}
P. Erd\H{o}s, C. Ko and R. Rado,
Intersection theorems for systems of finite sets,
Quart.\ J. Math. Oxford Ser.\ (2) 12 (1961) 313--320.

\bibitem{FW1986}
P. Frankl and R. M. Wilson,
The Erd\H{o}s--Ko--Rado theorem for vector spaces,
J. Combin.\ Theory Ser.\ A 43 (1986) 228--236.

\bibitem{F1999}
T-S.~Fu, Erd\H{o}s-Ko-Rado-type results over $J_q(n,d),H_q(n,d)$ and their designs, Discrete Math.\ 196 (1999), no. 1-3, 137--151.

\bibitem{GM2009}
C. Godsil, K. Meagher,
A new proof of the Erd\H{o}s-Ko-Rado theorem for intersecting families of permutations, 
European J. Combin.\ 30 (2009), no. 2, 404--414. 

\bibitem{HSS1999}
A.~S.~Hedayat, N.~J.~A.~Sloane, J.~Stufken, Orthogonal arrays, 
Theory and applications. With a foreword by C. R. Rao. Springer Series in Statistics. Springer-Verlag, New York, 1999. xxiv+416 pp.


\bibitem{H1975}
W. N. Hsieh,
Intersection theorems for systems of finite vector spaces,
Discrete Math.\ 12 (1975) 1--16.

\bibitem{H1987}
T.~Huang,
An analogue of the Erd\H{o}s--Ko--Rado theorem for the distance-regular graphs of bilinear forms,
Discrete Math.\ 64 (1987) 191--198.

\bibitem{KL2006}
C.~Y.~Ku, I.~Leader, An Erd\H{o}s-Ko-Rado theorem for partial permutations, 
Discrete Math.\ 306 (2006), no. 1, 74--86.

\bibitem{KR2008}
C.~Y.~Ku, D.~Renshaw, Erd\H{o}s-Ko-Rado theorems for permutations and set partitions, J. Combin.\ Theory Ser.\ A 115 (2008), 1008--1020.

\bibitem{M1982}
A.~Moon,
An analogue of the Erd\H{o}s--Ko--Rado theorem for the Hamming schemes $H(n,q)$,
J. Combin.\ Theory Ser.\ A 32 (1982) 386--390.

\bibitem{R1982}
B.~M.~I.~Rands, An extension of the Erd\H{o}s, Ko, Rado theorem to $t$-designs, 
J. Combin.\ Theory Ser.\ A.\ 32 (1982), no. 3, 391--395.

\bibitem{S1990}
H.~Suzuki, $2$-designs over ${\rm GF}(2^m)$, 
Graphs Combin.\ 6 (1990), no. 3, 293--296.

\bibitem{S1992}
H.~Suzuki, $2$-designs over ${\rm GF}(q)$,
Graphs Combin. 8 (1992), no. 4, 381--389. 

\bibitem{T2011}
H. Tanaka,
The Erd\H{o}s--Ko--Rado theorem for twisted Grassmann graphs,
to appear in Combinatorica;
arXiv:\href{http://arxiv.org/abs/1012.5692}{1012.5692}.

\bibitem{TAG1985}
H.~Tarnanen, M.~Aaltonen, J. J-M.~Goethals, On the nonbinary Johnson scheme, 
European J. Combin. 6 (1985), no. 3, 279--285.

\bibitem{T1996}
P.~Terwilliger,  Quantum matroids. Progress in algebraic combinatorics (Fukuoka, 1993), 323--441, Adv. Stud. Pure Math., 24, Math. Soc. Japan, Tokyo, 1996.

\bibitem{T1987}
S.~Thomas, Designs over finite fields, 
Geom.\ Dedicata 24 (1987), no. 2, 237--242.

\bibitem{Y2010}
L.~Yu-shuang, 
An Erd\H{o}s-Ko-Rado theorem for restricted signed sets, 
Acta Math. Appl. Sin. Engl. Ser. 26 (2010), no. 1, 107--112.

\bibitem{W1984}
R.~M.~Wilson,
The exact bound in the Erd\H{o}s--Ko--Rado theorem,
Combinatorica 4 (1984) 247--257.
\end{thebibliography}
\end{document}